\documentclass[10pt,reqno]{amsart}
\usepackage{bbm}
\usepackage{amsmath}
\usepackage{mathrsfs}
\usepackage{amsfonts} 
\usepackage{graphicx,latexsym,bm,amsmath,amssymb,verbatim,multicol,lscape}
\newtheorem{thm}{Theorem} [section]
\newtheorem{cor}[thm]{Corollary}
\newtheorem{lem}[thm]{Lemma}
\newtheorem{prop}[thm]{Proposition}
\theoremstyle{definition}

\theoremstyle{remark}

\numberwithin{equation}{section}

\begin{document} 
\title{Reversed Dickson polynomials of the fourth kind over finite fields}%
\author[Kaimin Cheng, Shaofang Hong and Xiaoer Qin]
{Kaimin Cheng$^{\rm a, b}$, Shaofang Hong$^{*, \rm a}$ and Xiaoer Qin$^{\rm c}$\\
$^{\rm a}$Mathematical College, Sichuan University, Chengdu 610064, P.R. China\\
$^{\rm b}$Department of Mathematics, Sichuan University Jinjiang College,
Pengshan 620860, P.R. China\\
$^{\rm c}$School of Mathematics and Statistics, Yangtze Normal
University, Chongqing 408100, P.R. China}
\thanks{$^*$Hong is the corresponding author and was supported
partially by National Science Foundation of China Grant \# 11371260.
Cheng was supported partially by the General Project of Department
of Education of Sichuan Province \# 15ZB0434. Qin was supported partially by
Science and Technology Research Projects of Chongqing Education Committee Grant
\# KJ15012004. \\
Emails: ckm20@126.com, cheng.km@stu.scu.edu.cn (K. Cheng);
sfhong@scu.edu.cn, s-f.hong@tom.com, hongsf02@yahoo.com (S. Hong); 
qincn328@sina.com (X. Qin).} 

\date{\today}%
\begin{abstract}
In this paper, we obtain several results on the permutational
behavior of the reversed Dickson polynomial $D_{n,3}(1,x)$
of the fourth kind over the finite field ${\mathbb F}_{q}$.
Particularly, we present the explicit evaluation of the
first moment $\sum_{a\in {\mathbb F}_{q}}D_{n,3}(1,a)$.
\end{abstract}

\maketitle

\section{Introduction}
Let ${\mathbb F}_{q}$ be the finite field of characteristic $p$ with $q$ elements.
Associated to any integer $n\ge 0$ and a parameter $a\in {\mathbb F}_{q}$,
the $n$-th \textit{Dickson polynomials of the first kind and of the second kind},
denoted by $D_n(x,a)$ and $E_n(x,a)$, are defined for $n\ge 1$ by
$$D_n(x,a):=\sum_{i=0}^{[\frac{n}{2} ]}\frac{n}{n-i}\binom{n-i}{i}(-a)^ix^{n-2i}$$
and
$$E_n(x,a):=\sum_{i=0}^{[\frac{n}{2} ]}\binom{n-i}{i}(-a)^ix^{n-2i},$$
respectively, and $D_0(x,a):=2, E_0(x,a):=1$, where $[\frac{n}{2}]$
means the largest integer no more than $\frac{n}{2}$. In 2012, Wang and Yucas
\cite{[WY]} further defined the \textit{$n$-th Dickson polynomial
of the $(k+1)$-th kind $D_{n,k}(x,a)\in {\mathbb F}_{q}[x]$} for $n\ge 1$ by
$$D_{n,k}(x,a):=\sum_{i=0}^{[\frac{n}{2} ]}\frac{n-ki}{n-i}\binom{n-i}{i}(-a)^ix^{n-2i}$$
and $D_{0,k}(x,a):=2-k$.

Hou, Mullen, Sellers and Yucas \cite{[HMSY]} introduced the definition of
\textit{the reversed Dickson polynomial of the first kind},
denoted by $D_n(a,x)$, as follows

$$D_n(a,x):=\sum_{i=0}^{[\frac{n}{2} ]}\frac{n}{n-i}\binom{n-i}{i}(-x)^ia^{n-2i}$$
if $n\ge 1$ and $D_0(a, x)=2$.
To extend the definition of reversed Dickson polynomials, Wang and Yucas \cite{[WY]}
defined \textit{the $n$-th reversed Dickson polynomial of
$(k+1)$-th kind} $D_{n,k}(a,x)\in {\mathbb F}_{q}[x]$, which is defined for $n\ge 1$ by
$$D_{n,k}(a,x):=\sum_{i=0}^{[\frac{n}{2} ]}\frac{n-ki}{n-i}\binom{n-i}{i}(-x)^ia^{n-2i}$$
and $D_{0,k}(a,x)=2-k$.

It is well known that $D_n(x,0)$ is a permutation polynomial of ${\mathbb F}_{q}$
if and only if $\gcd(n,q-1)=1$, and if $a\ne 0$, then $D_n(x,a)$ induces a permutation
of ${\mathbb F}_{q}$ if and only if $\gcd(n,q^2-1)=1$. Besides, there are lots of
published results on permutational properties of Dickson polynomial $E_n(x,a)$
of the second kind (see, for example, \cite{[Coh]}). In \cite{[WY]},
Wang and Yucas investigated the permutational properties of Dickson polynomial
$D_{n,2}(x,1)$ of the third kind. They got some necessary
conditions for $D_{n,2}(x,1)$ to be a permutation polynomial of ${\mathbb F}_{q}$.

Hou, Mullen, Sellers and Yucas \cite{[HMSY]} considered the permutational
behavior of reversed Dickson polynomial $D_{n}(a,x)$ of the first kind.
Actually, they showed that $D_n(a,x)$ is closely related to almost perfect
nonlinear functions, and obtained some families of permutation polynomials
from the revered Dickson polynomials of the first kind. In \cite{[HL]},
Hou and Ly found several necessary conditions for the revered Dickson
Polynomials $D_n(1,x)$ of the first kind to be a permutation polynomial.
Recently, Hong, Qin and Zhao \cite{[HQZ]} studied the revered Dickson
polynomial $E_n(a,x)$ of the second kind that is defined for $n\ge 1$ by
$$E_{n}(a,x):=\sum_{i=0}^{[\frac{n}{2} ]}\binom{n-i}{i}(-x)^ia^{n-2i}$$
and $E_{0}(a,x)=1$. In fact, they gave some necessary conditions for
the revered Dickson polynomial $E_n(1,x)$ of the second kind to be
a permutation polynomial of ${\mathbb F}_{q}$. Regarding the revered Dickson
polynomial $D_{n, 2}(a,x)\in {\mathbb F}_{q}[x]$ of the third kind,
from its definition one can derive that
\begin{equation}\label{c1}
D_{n,2}(a,x)=E_{n-1}(a,x)
\end{equation}
for each $x\in {\mathbb F}_{q}$. Using (\ref{c1}), we can deduce
immediately from \cite{[HQZ]} the similar results on the permutational
behavior of the reversed Dickson polynomial $D_{n,2}(a,x)$ of the third kind.
Actually, for the results in \cite{[HQZ]}, we need just to replace $E_n(1, x)$
by $D_{n, 2}(1, x)$ and replace all other $n$ by $n-1$, then we can
obtain the corresponding results on the revered Dickson
polynomial $D_{n, 2}(a,x)\in {\mathbb F}_{q}[x]$ of the third kind.
We here do not list these results.

In this paper, our main goal is to investigate the revered
Dickson polynomial $D_{n,3}(a,x)$ of the fourth kind which
is defined by
\begin{equation}\label{c2}
D_{n,3}(a,x):=\sum_{i=0}^{[\frac{n}{2} ]}
\frac{n-3i}{n-i}\binom{n-i}{i}(-x)^ia^{n-2i}
\end{equation}
if $n\ge 1$ and $D_{0,3}(a,x):=-1$. For $a\ne 0$, we write $x=y(a-y)$
with an indeterminate $y\ne \frac{a}{2}$. Then $D_{n,3}(a,x)$ can be 
rewritten as

\begin{equation}\label{c3}
D_{n,3}(a,x)=\frac{(2a-y)y^n-(y+a)(a-y)^n}{2y-a}.
\end{equation}
We have
$$D_{n, 3}\Big(a, \frac{a^2}{4}\Big)=\frac{(3n-1)a^n}{2^n}.\eqno(1.4)$$
In fact, (1.3) and (1.4) follows from Theorem 2.2 (i)
and Theorem 2.4 (i) below. It is easy to see that $D_{n,3}(a,x)=E_n(a,x)$
if ${\rm char}({\mathbb F}_{q})=2$, and $D_{n,3}(a,x)=D_{n}(a,x)$
if ${\rm char}({\mathbb F}_{q})=3$. Thus we always
assume $p={\rm char}({\mathbb F}_{q})>3$ in what follows.

The paper is organized as follows. First in section 2, we study the
properties of the reversed Dickson polynomial $D_{n,3}(a,x)$ of the
fourth kind. Subsequently, in Section 3, we prove a necessary condition
for the reversed Dickson polynomial $D_{n,3}(1,x)$ of the fourth kind
to be a permutation polynomial of ${\mathbb F}_{q}$ and then
introduce an auxiliary polynomial to present a characterization
for $D_{n,3}(1,x)$ to be a permutation of ${\mathbb F}_q$. 
From the Hermite criterion \cite{[LN]} one knows that a function
$f: {\mathbb F}_{q}\rightarrow {\mathbb F_{q}}$
is a permutation polynomial of ${\mathbb F_{q}}$
if and only if the $i$-th moment
\begin{align*}\nonumber
\sum_{a\in {\mathbb F_{q}}}f(a)^i=
{\left\{\begin{array}{rl}
0,& {\rm if} \ 0\le i\le q-2,\\
-1,& {\rm if} \ i=q-1.
\end{array}
\right.}
\end{align*}
Thus to understand well the permutational behavior of
the reversed Dickson polynomial $D_{n,3}(1,x)$
of the fourth kind, we would like to know if the $i$-th
moment $\sum_{a\in {\mathbb F}_q}D_{n,3}(1,a)^i$
is computable. We are able to treat with this sum when $i =1$.
The final section is devoted to the computation of the first moment
$\sum_{a\in {\mathbb F_{q}}}D_{n,3}(1,a)$.

\section{Revered Dickson polynomials of the fourth kind}

In this section, we study the properties of the revered Dickson polynomials
$D_{n,3}(a,x)$ of the fourth kind. Clearly, if $a=0$, then
\begin{align*}\nonumber
D_{n,3}(0,x)={\left\{
  \begin{array}{rl}
0,& {\rm if} \ n \ {\rm is \ odd},\\
(-1)^{\frac{n}{2}+1}x^{\frac{n}{2}},& {\rm if} \ n \ {\rm is \ even}.
\end{array}
\right.}
\end{align*}
Therefore, $D_{n,3}(0,x)$ is a PP (permutation polynomial) of ${\mathbb F_{q}}$
if and only if $n$ is an even integer with $\gcd(\frac{n}{2},q-1)=1$.
In what follows, we always let $a\in {\mathbb F_{q}^*}$.
First, we give a basic fact as follows.
\begin{lem}
\cite{[LN]} Let $f(x)\in {\mathbb F_{q}}[x]$. Then $f(x)$
is a PP of ${\mathbb F_{q}}$ if and only if
$cf(dx)$ is a PP of ${\mathbb F_{q}}$ for any
given $c, d\in {\mathbb F_{q}^*}$.
\end{lem}
Then we can deduce the following result.
\begin{thm}
Let $a, b\in {\mathbb F_{q}^*}$. Then the following are true.

{\rm (i)}. One has $D_{n,3}(a,x)=\frac{a^n}{b^n}D_{n,3}(b,\frac{b^2}{a^2}x)$.

{\rm (ii)}. We have that $D_{n,3}(a,x)$ is a PP of ${\mathbb F_{q}}$
if and only if $D_{n,3}(1,x)$ is a PP of ${\mathbb F_{q}}$.
\end{thm}

\begin{proof}
{\rm (i)}. By the definition of $D_{n,3}(a,x)$, we have

\begin{align*}
&\frac{a^n}{b^n}D_{n,3}\Big(b,\frac{b^2}{a^2}x\Big)\\
=&\frac{a^n}{b^n}
\sum_{i=0}^{[\frac{n}{2} ]}\frac{n-3i}{n-i}\binom{n-i}{i}
(-1)^ib^{n-2i}\frac{b^{2i}}{a^{2i}}x^{i}\\
=&\sum_{i=0}^{[\frac{n}{2}]}\frac{n-3i}{n-i}\binom{n-i}{i}(-1)^ia^{n-2i}x^{i}\\
=&D_{n,3}(a,x)
\end{align*}
as required. Part (i) is proved.

(ii). Taking $b=1$ in part (i), we have
$$D_{n,3}(a,x)=a^nD_{n,3}\Big(1,\frac{x}{a^2}\Big).$$
It then follows from Lemma 2.1 that $D_{n,3}(a,x)$ is a PP
of ${\mathbb F_{q}}$ if and only if
$D_{n,3}(1,x)$ is a PP of ${\mathbb F_{q}}$.
This completes the proof of part (ii). So Theorem 2.2 is proved.
\end{proof}

Theorem 2.2 tells us that to study the permutational
behavior of $D_{n,3}(a,x)$ over ${\mathbb F_{q}}$,
one only needs to consider that of $D_{n,3}(1,x)$.
In the following, we supply several basic properties
on the revered Dickson polynomial $D_{n,3}(1,x)$
of the fourth kind. The following result is given
in \cite{[HQZ]} and \cite{[HMSY]}
without proof. For the completeness,
we here present a proof.

\begin{lem} \cite{[HQZ]} \cite{[HMSY]}
Let $n\ge 0$ be an integer. Then
we have $D_n(1,x(1-x))=x^n+(1-x)^n$ and
$E_{n}(1,x(1-x))=\frac{x^{n+1}-(1-x)^{n+1}}{2x-1}$.
\end{lem}
\begin{proof}
Since $D_0(1, x(1-x))=2$, the first formula is true for the
case $n=0$. Let now $n\ge 1$ be an integer. Then
$$D_n(1,x(1-x))=\sum_{i=0}^{[\frac{n}{2}]}
\frac{n}{n-i}\binom{n-i}{i}(x+1-x)^{n-2i}(-x(1-x))^{i}.$$
It then follows from Waring's formula (see, for instance,
Theorem 1.76 of \cite{[LN]}) that for any integer $n\ge 1$,
we have
\begin{align}
D_n(1,x(1-x))=x^n+(1-x)^n\label{ch0}.
\end{align}
as desired. The first formula is proved.

Since $E_0(1, x(1-x))=E_1(1, x(1-x))=1$, the second formula holds
when $n=0$ and 1. Now let $n\ge 2$ be an integer. Then we have

\begin{align*}
&E_n(1,x(1-x))\\
=&\sum_{i=0}^{[\frac{n}{2}]}
\frac{n-i}{n-i}\binom{n-i}{i}(-x(1-x))^{i}\nonumber\\
=&\sum_{i=0}^{[\frac{n}{2}]}
\frac{n}{n-i}\binom{n-i}{i}(-x(1-x))^{i}+x(1-x)\sum_{i=1}^{[\frac{n}{2}]}
\frac{i}{n-i}\binom{n-i}{i}(-x(1-x))^{i-1}\nonumber\\
=&\sum_{i=0}^{[\frac{n}{2}]}
\frac{n}{n-i}\binom{n-i}{i}(-x(1-x))^{i}+x\sum_{i=1}^{[\frac{n}{2}]}
\binom{n-1-i}{i-1}(-x(1-x))^{i-1}\nonumber\\
=&\sum_{i=0}^{[\frac{n}{2}]}
\frac{n}{n-i}\binom{n-i}{i}(-x(1-x))^{i}+x\sum_{i=0}^{[\frac{n-2}{2}]}
\binom{n-2-i}{i}(-x(1-x))^{i}\nonumber\\
=&D_n(1,x(1-x))+xE_{n-2}(1,x(1-x)).
\end{align*}
It follows that
\begin{align}
&E_n(1, x(1-x))\nonumber \\
=&\sum_{i=0}^{[\frac{n}{2}]-1}x^i(1-x)^iD_{n-2i}(1, x(1-x))
+x^{[\frac{n}{2}]}(1-x)^{[\frac{n}{2}]}E_{n-2[\frac{n}{2}]}(1, x(1-x)).  \label{ch00}
\end{align}
From (\ref{ch0}) and (\ref{ch00}) one can deduce
that if $n\ge 1$ is odd, then we have
\begin{align*}
&E_n(1, x(1-x))\\
=&\sum_{i=0}^{\frac{n-3}{2}}x^i(1-x)^i D_{n-2i}(1, x(1-x))+
y^{\frac{n-1}{2}}(1-x)^{\frac{n-1}{2}}E_{1}(1, x(1-x))\nonumber\\
=&\sum_{i=0}^{\frac{n-3}{2}}x^i(1-x)^i \big(x^{n-2i}+(1-x)^{n-2i}\big)+
x^{\frac{n-1}{2}}(1-x)^{\frac{n-1}{2}}(x+1-x)\nonumber\\
=&\sum_{i=0}^{\frac{n-1}{2}}\big(x^{n-i}(1-x)^i+x^i(1-x)^{n-i}\big)\nonumber\\
=&\sum_{i=0}^{n}x^{n-i}(1-x)^i\nonumber\\
=&\frac{x^{n+1}-(1-x)^{n+1}}{2x-1}, \nonumber\\
\end{align*}
and if $n\ge 0$ is even, then one has

\begin{align*}
&E_n(1, x(1-x))\\
=&\sum_{i=0}^{\frac{n}{2}-1}x^i(1-x)^i D_{n-2i}(1, x(1-x))+
x^{\frac{n}{2}}(1-x)^{\frac{n}{2}}E_0(1, x(1-x))\nonumber\\
=&\sum_{i=0}^{\frac{n}{2}-1}\big(x^{n-i}(1-x)^i+x^i(1-x)^{n-i}\big)
+x^{\frac{n}{2}}(1-x)^{\frac{n}{2}}\nonumber\\
=&\sum_{i=0}^{n}x^{n-i}(1-x)^i\nonumber\\
=&\frac{x^{n+1}-(1-x)^{n+1}}{2x-1} \nonumber\\
\end{align*}
as expected. So the second formula is proved.

This concludes the proof of Lemma 2.3.
\end{proof}

\begin{thm}
Each of the following is true.

{\rm (i)}. For any integer $n\ge 0$,
we have $D_{n,3}(1,\frac{1}{4})=\frac{3n-1}{2^n}$
and
$D_{n,3}(1,x(1-x))=\frac{(2-x)x^n-(x+1)(1-x)^n}{2x-1}$
if $x\ne \frac{1}{2}$.

{\rm (ii)}. If $n_1$ and $n_2$ are positive integers such that
$n_1\equiv n_2\pmod{q^2-1}$, then one has
$D_{n_1,3}(1,x_0)=D_{n_2,3}(1,x_0)$ for any
$x_0\in {\mathbb F_{q}\setminus\{\frac{1}{4}\}}$.
\end{thm}
\begin{proof}
(i). First of all, it is easy to see that
$D_{0,3}\big(1,\frac{1}{4}\big)=-1=\frac{3\times 0-1}{2^0}$
and
$D_{1,3}\big(1,\frac{1}{4}\big)=1=\frac{3\times 1-1}{2^1}$.
the first identity is true for the cases that $n=0$ and 1.
Now let $n\ge 2$. Then one has
\begin{align*}
D_{n,3}\Big(1,\frac{1}{4}\Big)&=\sum_{i=0}^{[\frac{n}{2}]}
\frac{n-3i}{n-i}\binom{n-i}{i}\Big(-\frac{1}{4}\Big)^{i}\\
&=\sum_{i=0}^{[\frac{n}{2} ]}\frac{n-2i}{n-i}\binom{n-i}{i}
\Big(-\frac{1}{4}\Big)^{i}+\sum_{i=0}^{[\frac{n}{2} ]}
\frac{-i}{n-i}\binom{n-i}{i}\Big(-\frac{1}{4}\Big)^{i}\\
&=D_{n,2}\Big(1,\frac{1}{4}\Big)+\frac{1}{4}\sum_{i=0}^{[\frac{n}{2} ]-1}
\binom{n-2-i}{i}\Big(-\frac{1}{4}\Big)^{i}\\
&=D_{n,2}\Big(1,\frac{1}{4}\Big)+\frac{1}{4}E_{n-2}\Big(1,\frac{1}{4}\Big).
\end{align*}
But (\ref{c1}) gives us that
$D_{n,2}(1,\frac{1}{4})=E_{n-1}(1, \frac{1}{4})$.
Hence Theorem 2.2 of \cite{[HQZ]} implies that
\begin{align*}
D_{n,3}\Big(1,\frac{1}{4}\Big)=&E_{n-1}\Big(1,\frac{1}{4}\Big)
+\frac{1}{4}E_{n-2}\Big(1,\frac{1}{4}\Big)\\
=&\frac{n}{2^{n-1}}+\frac{1}{4}\cdot\frac{n-1}{2^{n-2}}\\
=&\frac{3n-1}{2^n}
\end{align*}
as desired. So the first identity is proved.

Now we turn our attention to the second identity.
Let $x\ne \frac{1}{2}$. Then by the definition of
the $n$-th reversed Dickson polynomial
of the fourth kind, one has
\begin{align}
D_{n,3}(1,x(1-x))&=\sum_{i=0}^{[\frac{n}{2}]}
\frac{n-3i}{n-i}\binom{n-i}{i}(-x(1-x))^{i}\nonumber\\
&=\sum_{i=0}^{[\frac{n}{2}]}
\frac{3(n-i)-2n}{n-i}\binom{n-i}{i}(-x(1-x))^{i}\nonumber\\
&=3\sum_{i=0}^{[\frac{n}{2}]}
\binom{n-i}{i}(-x(1-x))^{i}-2\sum_{i=0}^{[\frac{n}{2}]}
\frac{n}{n-i}\binom{n-i}{i}(-x(1-x))^{i}\nonumber\\
&=3E_n(1,x(1-x))-2D_n(1,x(1-x))\label{ch1}.
\end{align}
But Lemma 2.3 gives us that
\begin{align}
D_n(1,x(1-x))=x^n+(1-x)^n\label{ch3}
\end{align}
and
\begin{align}
E_n(1,x(1-x))=\sum_{i=0}^{n}x^{n-i}(1-x)^i
=\frac{x^{n+1}-(1-x)^{n+1}}{2x-1}   \label{ch4}.
\end{align}
Thus it follows from (\ref{ch1}) to (\ref{ch4}) that
\begin{align*}
D_{n,3}(1, x)=&D_{n,3}(1,y(1-y))\\
=&3E_n(1,y(1-y))-2D_n(1,y(1-y))\\
=&\frac{3y^{n+1}-3(1-y)^{n+1}}{2y-1}-2\big(y^n+(1-y)^n\big)\\
=&\frac{(2-y)y^n-(y+1)(1-y)^n}{2y-1}
\end{align*}
as required. So the second identity holds. Part (i) is proved.

(ii). For each $x_0\in {\mathbb F_{q}\setminus\{\frac{1}{4}\}}$,
one can choose an element $y_0\in {\mathbb F_{q^2}\setminus\{\frac{1}{2}\}}$
such that $x_0=y_0(1-y_0)$. Since $n_1\equiv n_2\pmod{q^2-1}$,
one has $y_0^{n_1}=y_0^{n_2}$ and $(1-y_0)^{n_1}=(1-y_0)^{n_2}$.
It then follows from part (i) that
\begin{align*}
D_{n_1,3}(1,x_0)&=D_{n_1,3}(1,y_0(1-y_0))\\
&=\frac{(2-y_0)y_0^{n_1}-(y_0+1)(1-y_0)^{n_1}}{2y_0-1}\\
&=\frac{(2-y_0)y_0^{n_2}-(y_0+1)(1-y_0)^{n_2}}{2y_0-1}\\
&=D_{n_2,3}(1,x_0)
\end{align*}
as desired. This ends the proof of Theorem 2.4.
\end{proof}

Evidently, by Theorem 2.2 (i) and Theorem 2.4 (i) one can derive
that (1.3) and (1.4) are true.

\begin{prop}
Let $n\ge 2$ be an integer. Then the recursion
$$D_{n,3}(1,x)=D_{n-1,3}(1,x)-xD_{n-2,3}(1,x)$$
holds for any $x\in {\mathbb F_{q}}$.
\end{prop}
\begin{proof} We consider the following two cases.

{\sc Case 1.} $x\ne \frac{1}{4}$. For this case, one may let $x=y(1-y)$
with $y\in {\mathbb F_{q^2}}\setminus \{ \frac{1}{2}\}$.
Then by Theorem 2.4 (i), we have
\begin{align*}
&\ \ \ \ \ D_{n-1,3}(1,x)-xD_{n-2,3}(1,x)\\
&=D_{n-1,3}(1,y(1-y))-y(1-y)D_{n-2,3}(1,y(1-y))\\
&=\frac{(2-y)y^{n-1}-(y+1)(1-y)^{n-1}}{2y-1}-y(1-y)\frac{(2-y)y^{n-2}-(y+1)(1-y)^{n-2}}{2y-1}\\
&=\frac{(2-y)y^{n}-(y+1)(1-y)^{n}}{2y-1}\\
&=D_{n,3}(1,x)
\end{align*}
as required.

{\sc Case 2.} $x=\frac{1}{4}$. Then by Theorem 2.4 (i), we have
\begin{align*}
&D_{n-1,3}\Big(1,\frac{1}{4}\Big)-\frac{1}{4}D_{n-2,3}\Big(1,\frac{1}{4}\Big)\\
=&\frac{3n-4}{2^{n-1}}-\frac{1}{4}\frac{3n-7}{2^{n-2}}\\
=&\frac{3n-1}{2^{n}}\\
=&D_{n,3}\Big(1,\frac{1}{4}\Big).
\end{align*}

This concludes the proof of Proposition 2.5.
\end{proof}

By Proposition 2.5, we can obtain the generating function of
the revered Dickson polynomial $D_{n,3}(1,x)$ of the fourth kind as follows.
\begin{prop}
The generating function of $D_{n,3}(1,x)$ is given by
$$\sum_{n=0}^{\infty}D_{n,3}(1,x)t^n=\frac{2t-1}{1-t+xt^2}.$$
\end{prop}
\begin{proof}
By the recursion presented in Proposition 2.5, we have
\begin{align*}
&(1-t+xt^2)\sum_{n=0}^{\infty}D_{n,3}(1,x)t^n\\
=&\sum_{n=0}^{\infty}D_{n,3}(1,x)t^n-\sum_{n=0}
^{\infty}D_{n,3}(1,x)t^{n+1}+x\sum_{n=0}^{\infty}D_{n,3}(1,x)t^{n+2}\\
=&2t-1+\sum_{n=0}^{\infty}\big(D_{n+2,3}
(1,x)-D_{n+1,3}(1,x)+xD_{n,3}(1,x)\big)t^{n+2}\\
=&2t-1.
\end{align*}
Thus the desired result follows immediately.
\end{proof}

Now we can use Theorem 2.4 to present an explicit formula for $D_{n, 3}(1, x)$
when $n$ is a power of the characteristic $p$. Then we show that $D_{n, 3}(1, x)$
is not a PP of ${\mathbb F}_q$ in this case.

\begin{prop}
Let $p={\rm char}({\mathbb F_{q}})>3$ and $k$ be a positive integer. Then
$$2^{p^k}D_{p^k,3}(1,x)+1=3(1-4x)^{\frac{p^k-1}{2}}.$$
\end{prop}
\begin{proof}
Putting $x=y(1-y)$ in Theorem 2.4 (i) gives us that
\begin{align*}
D_{p^k,3}(1,x)=&D_{p^k,3}(1,y(1-y))\\
=&\frac{(2-y)y^{p^k}-(y+1)(1-y)^{p^k}}{2y-1}\\
=&\frac{\frac{3-u}{2}\big(\frac{u+1}{2}\big)^{p^k}
-\frac{3+u}{2}\big(\frac{1-u}{2}\big)^{p^k}}{u}\\
=&\frac{1}{2^{p^k+1}u}\Big((3-u)(u+1)^{p^k}-(u+3)(1-u)^{p^k}\Big)\\
=&\frac{1}{2^{p^k}}(3u^{p^k-1}-1),
\end{align*}
where $u=2y-1$. So we obtain that
\begin{align*}
&2^{p^k}D_{p^k,3}(1,x)\\
=&3(u^2)^{\frac{p^k-1}{2}}-1\\
=&3\big((2y-1)^2\big)^{\frac{p^k-1}{2}}-1,
\end{align*}
which infers that
$$2^{p^k}D_{p^k,3}(1,x)+1=3(1-4x)^{\frac{p^k-1}{2}}$$
as desired. So Proposition 2.7 is proved.
\end{proof}

It is well known that every linear polynomial over
${\mathbb F_{q}}$ is a PP of ${\mathbb F_{q}}$
and that the monomial $x^n$ is a PP of
${\mathbb F_{q}}$ if and only if $\gcd(n,q-1)=1$. Then by
Proposition 2.7, we have the following result.
\begin{cor}
Let $p>3$ be a prime and $q=p^e$. Let $e$ and $k$ be positive integers with
$k\le e$. Then $D_{p^k,3}(1,x)$ is not a PP of ${\mathbb F_{q}}$.
\end{cor}
\begin{proof}
By Proposition 2.7, we know that $D_{p^k,3}(1,x)$
is a PP of ${\mathbb F_{q}}$ if and only if
$$(1-4x)^{\frac{p^k-1}{2}}$$
is a PP of ${\mathbb F_{q}}$ which is equivalent to
$$\gcd\Big(\frac{p^k-1}{2},q-1\Big)=1.$$
The latter one is impossible since
$\frac{p-1}{2}|\gcd\big(\frac{p^k-1}{2},q-1\big)$ implies that
$$\gcd\Big(\frac{p^k-1}{2},q-1\Big)\ge\frac{p-1}{2}>1.$$
Thus $D_{p^k,3}(1,x)$ is not a PP of ${\mathbb F_{q}}$.
\end{proof}
\begin{lem}\cite{[HMSY]}
Let $x\in{\mathbb F_{q^2}}$. Then $x(1-x)\in{\mathbb F_{q}}$
if and only if $x^q=x$ or $x^q=1-x$.
\end{lem}

Let $V$ be defined by
$$V:=\{x\in {\mathbb F_{q^2}}: x^q=1-x\}.$$
Clearly, ${\mathbb F_{q}}\cap V=\{\frac{1}{2}\}$. Then we obtain a characterization for
$D_{n,3}(1,x)$ to be a PP of ${\mathbb F_{q}}$ as follows.
\begin{thm}
Let $q=p^e$ with $p>3$ being a prime and $e$ being a positive integer. Let
$$f:y\mapsto\frac{(2-y)y^n-(y+1)(1-y)^n}{2y-1}$$
be a mapping on $({\mathbb F_{q}}\cup V)\setminus \{\frac{1}{2}\}$. Then $D_{n,3}(1,x)$ is
a PP of ${\mathbb F_{q}}$ if and only if $f$ is $2$-to-$1$ and $f(y)\ne\frac{3n-1}{2^n}$
for any $y\in({\mathbb F_{q}}\cup V)\setminus \{\frac{1}{2}\}$.
\end{thm}
\begin{proof}
First, we show the sufficiency part. Let $f$ be $2$-to-$1$ and $f(y)\ne\frac{3n-1}{2^n}$
for any $y\in({\mathbb F_{q}}\cup V)\setminus \{\frac{1}{2}\}$.
Let $D_{n,3}(1,x_1)=D_{n,3}(1,x_2)$ for $x_1,x_2\in{\mathbb F_{q}}$. To show that
$D_{n, 3}(1, x)$ is a PP of ${\mathbb F_{q}}$, it suffices to show that $x_1=x_2$
that will be done in what follows.

First of all, one can find $y_1,y_2\in{\mathbb F_{q^2}}$ satisfying
$x_1=y_1(1-y_1)$ and $x_2=y_2(1-y_2)$.
By Lemma 2.9, we know that $y_1,y_2\in{\mathbb F_{q}}\cup V$.
We divide the proof into the following two cases.

{\sc Case 1.} At least one of $x_1$ and $x_2$ is equal to $\frac{1}{4}$.
Without loss of any generality, we may let $x_1=\frac{1}{4}$.
So by Theorem 2.4 (i), one derives that
\begin{equation}\label{c4}
D_{n,3}(1,x_2)=D_{n,3}(1,x_1)=D_{n,3}\Big(1,\frac{1}{4}\Big)=\frac{3n-1}{2^n}.
\end{equation}
We claim that $x_2=\frac{1}{4}$. Assume that $x_2\ne\frac{1}{4}$.
Then $y_2\ne \frac{1}{2}$. Since $f(y)\ne\frac{3n-1}{2^n}$
for any $y\in({\mathbb F_{q}}\cup V)\setminus \{\frac{1}{2}\}$,
by Theorem 2.4 (i), we get that
$$D_{n,3}(1,x_2)=\frac{(2-y_2)y_2^n-(y_2+1)(1-y_2)^n}{2y_2-1}=f(y_2)\ne \frac{3n-1}{2^n},$$
which contradicts to (\ref{c4}). Hence the claim is true, and so we have $x_1=x_2$ as required.

{\sc Case 2.} Both of $x_1$ and $x_2$ are not equal to $\frac{1}{4}$. Then
$y_1\ne \frac{1}{2}$ and $y_2\ne \frac{1}{2}$. Since $D_{n,3}(1,x_1)=D_{n,3}(1,x_2)$,
by Theorem 2.4 (i), one has
$$\frac{(2-y_2)y_1^n-(y_1+1)(1-y_1)^n}{2y_1-1}=\frac{(2-y_2)
y_2^n-(y_2+1)(1-y_2)^n}{2y_2-1},$$
which is equivalent to $f(y_1)=f(y_2)$. However, $f$ is a
$2$-to-$1$ mapping on $({\mathbb F_{q}}\cup V)\setminus \{\frac{1}{2}\}$, and $f(y_2)=f(1-y_2)$ by
the definition of $f$. It then follows that $y_1=y_2$ or $y_1=1-y_2$. Thus $x_1=x_2$ as desired.
Hence the sufficiency part is proved.

Now we prove the necessity part. Let $D_{n,3}(1,x)$ be a PP of ${\mathbb F_{q}}$.
Choose two elements $y_1, y_2\in({\mathbb F_{q}}\cup V)\setminus \{\frac{1}{2}\}$
such that $f(y_1)=f(y_2)$, that is,
\begin{equation}\label{c5}
\frac{(2-y_1)y_1^n-(y_1+1)(1-y_1)^n}{2y_1-1}=\frac{(2-y_2)
y_2^n-(y_2+1)(1-y_2)^n}{2y_2-1}.
\end{equation}
Since $y_1, y_2\in({\mathbb F_{q}}\cup V)\setminus \{\frac{1}{2}\}$,
it follows from Lemma 2.9 that $y_1(1-y_1)\in{\mathbb F_{q}}$ and
$y_2(1-y_2)\in{\mathbb F_{q}}$. So by Theorem 2.4 (i), (\ref{c5}) implies that
$$D_{n,3}(1,y_0(1-y_0))=D_{n,3}(1,y(1-y)).$$
Thus $y_1(1-y_1)=y_2(1-y_2)$ since $D_{n,3}(1,x)$ is a PP of ${\mathbb F_{q}}$,
which infers that $y_1=y_2$ or $y_1=1-y_2$. Since $y_2\ne \frac{1}{2}$,
one has $y_2\ne 1-y_2$. Therefore $f$ is a
$2$-to-$1$ mapping on $({\mathbb F_{q}}\cup V)\setminus \{\frac{1}{2}\}$.

Now take $y'\in ({\mathbb F_{q}}\cup V)\setminus \{\frac{1}{2}\}$.
Then from Lemma 2.9 it follows that $y'(1-y')\in{\mathbb F_{q}}$ and
$$y'(1-y')\ne \frac{1}{2}\Big(1-\frac{1}{2}\Big).$$
Notice that $D_{n,3}(1,x)$ is a PP of ${\mathbb F_{q}}$. Hence one has
$$D_{n,3}(1,y'(1-y'))\ne D_{n,3}\Big(1,\frac{1}{2}\Big(1-\frac{1}{2}\Big)\Big).$$
But Theorem 2.4 (i) tells us that
$$
D_{n,3}\Big(1,\frac{1}{2}\Big(1-\frac{1}{2}\Big)\Big)=\frac{3n-1}{2^n}.
$$
Then by Theorem 2.4 (i) and noting that $y'\ne \frac{1}{2}$, we have
$$\frac{(2-y')y'^n-(y'+1)(1-y')^n}{2y'-1}\ne \frac{3n-1}{2^n},$$
which infers that $f(y')\ne \frac{3n-1}{2^n}$ for any
$y'\in ({\mathbb F_{q}}\cup V)\setminus \{\frac{1}{2}\}$.
So the necessity part is proved.

The proof of Theorem 2.10 is complete.
\end{proof}

\section{A necessary condition for $D_{n,3}(1,x)$ to be permutational and an auxiliary polynomial}

In this section, we study some necessary conditions on $n$ for $D_{n,3}(1,x)$
to be a PP of ${\mathbb F_{q}}$. It is easy to check that
$$D_{n,3}(1,0)=1, D_{0,3}(1,1)=-1, D_{1,3}(1,1)=1, D_{0,3}(1,-2)=-1, D_{1,3}(1,-2)=1.$$
Then by Proposition 2.5, we have the following recursion relations
$$
{\left\{
  \begin{array}{ll}
D_{0,3}(1,1)=-1,\\
D_{1,3}(1,1)=1,\\
D_{n+2,3}(1,1)=D_{n+1,3}(1,1)-D_{n,3}(1,1),
\end{array}
\right.}
$$
and
$${\left\{
  \begin{array}{ll}
D_{0,3}(1,-2)=-1,\\
D_{1,3}(1,-2)=1,\\
D_{n+2,3}(1,-2)=D_{n+1,3}(1,-2)+2D_{n,3}(1,-2).
\end{array}
\right.}
$$
From these recursive formulas, one can easily show that the sequences
$$\{D_{n,3}(1,1)|n\in \mathbb{N}\} \ {\rm and} \ \{D_{n,3}(1,-2)|n\in \mathbb{N}\}$$
are periodic with the smallest periods 6 and 2, respectively. In fact, one has
$$
D_{n,3}(1,1)={\left\{
  \begin{array}{rl}
1,& {\rm if} \ n\equiv 1,3\pmod{6},\\
-1,& {\rm if} \ n\equiv 0,4\pmod{6},\\
2,& {\rm if} \ n\equiv 2\pmod{6},\\
-2,& {\rm if} \ n\equiv 5\pmod{6}
\end{array}
\right.}
$$
and
$$
D_{n,3}(1,-2)={\left\{
  \begin{array}{rl}
1,& {\rm if} \ n\equiv 1\pmod{2},\\
-1,& {\rm if} \ n\equiv 0\pmod{2}.
\end{array}
\right.}
$$
\begin{thm}
Assume that $D_{n,3}(1,x)$ is a PP of ${\mathbb F_{q}}$ with $q=p^e$ and $p>3$. Then
$n\equiv 2 \pmod{6}$.
\end{thm}
\begin{proof}
Let $D_{n,3}(1,x)$ be a PP of ${\mathbb F_{q}}$. Then $D_{n,3}(1,0), D_{n,3}(1,1)$
and $D_{n,3}(1,-2)$ are distinct. Since $D_{n,3}(1,0)=1$, one has $D_{n,3}(1,1)\ne 1$
and $D_{n,3}(1,-2)\ne 1$. Then the above results tells us that $n\not\equiv 1,3,5 \pmod 6$.
Further, we have $D_{n,3}(1,-2)=-1$ which means that $n$ must be even, and so $D_{n,3}(1,1)\ne -2$.
But $D_{n,3}(1,1)\ne D_{n,3}(1,-2)$. So $D_{n,3}(1,1)\ne -1$. Hence $D_{n,3}(1,1)=2$.
Finally, the desired result $n\equiv 2 \pmod{6}$ follows immediately.
\end{proof}

Evidently, Corollary 2.8 can be easily deduced from Theorem 3.1.
Furthermore, By Theorem 3.1, we know that $D_{n,3}(1,x)$
is not a PP of ${\mathbb F_{q}}$ if $n$ is odd.

In what follows, we investigate $D_{n,3}(1,x)$ with $n$
being an even number. We define the following
auxiliary polynomial $f_{n}(x)\in \mathbb{Z}[x]$ by
$$f_{n}(x):=-x^{\frac{n}{2}}+\sum_{j=0}^{\frac{n}{2}-1}
\frac{3n-8j-1}{n+1}\binom{n+1}{2j+1}x^j.$$
Then we have the following relation between $D_{n,3}(1,x)$ and $f_{n}(x)$.
\begin{thm}
Let $p>3$ be a prime and $n\ge 0$ be an even integer. Then

{\rm (i).} One has
$$D_{n,3}(1,x)=\frac{1}{2^n}f_{n}(1-4x). \eqno(3.1)$$

{\rm (ii).} We have that $D_{n,3}(1,x)$ is a PP of ${\mathbb F_{q}}$
if and only if $f_n(x)$ is a PP of ${\mathbb F_{q}}$.
\end{thm}
\begin{proof} (i). First, let $x\in{\mathbb F_{q}}\setminus\{\frac{1}{4}\}$. Then there exists
$y\in{\mathbb F_{q^2}}\setminus\{\frac{1}{2}\}$ such that $x=y(1-y)$. Let $u=2y-1$.
Since for any integer $j$ with $0\le j\le \frac{n}{2}-1$, one has
$$
3\binom{n}{2j+1}-\binom{n}{2j}=\frac{3n-8j-1}{n+1}\binom{n+1}{2j+1},
$$
it then follows from Theorem 2.4 (i) that
\begin{align*}
D_{n,3}(1,x)&=D_{n,3}(1,y(1-y))\\
&=\frac{(2-y)y^{n}-(y+1)(1-y)^{n}}{2y-1}\\
&=\frac{\frac{3-u}{2}\big(\frac{u+1}{2}\big)^{n}-\frac{3+u}{2}\big(\frac{1-u}{2}\big)^{n}}{u}\\
&=\frac{1}{2^{n+1}u}\big((3-u)(u+1)^{n}-(u+3)(1-u)^{n}\big)\\
&=\frac{1}{2^{n}}\Big(-u^{n}+\sum_{j=0}^{\frac{n}{2}-1}\Big(3\binom{n}{2j+1}-\binom{n}{2j}\Big)u^{2j}\Big)\\
&=\frac{1}{2^n}f_n(u^2)\\
&=\frac{1}{2^n}f_n(1-4y(1-y))\\
&=\frac{1}{2^n}f_n(1-4x)
\end{align*}
as desired. So (3.1) holds in this case.

Consequently, we let $x=\frac{1}{4}$. Then by Theorem 2.4 (i),
we have
$$D_{n,3}\Big(1,\frac{1}{4}\Big)=\frac{3n-1}{2^n}.$$
On the other hand, we can easily check that $f_n(0)=3n-1$. Therefore
$$D_{n,3}\Big(1,\frac{1}{4}\Big)=\frac{1}{2^n}f_n(0)
=\frac{1}{2^n}f_n\Big(1-4\times \frac{1}{4}\Big)$$
as one desires. So (3.1) is proved.

(ii).
Notice that $\frac{1}{2^n}\in{\mathbb F_{q}^*}$ and $1-4x$
is linear. So $D_{n,3}(1,x)$ is a PP of ${\mathbb F_{q}}$ if and only if
$f_n(x)$ is a PP of ${\mathbb F_{q}}$. This ends the proof of Theorem 3.2.
\end{proof}
\section{The first moment $\sum_{a\in {\mathbb F_{q}}}D_{n,3}(1,a)$}

In this section, we compute the first moment
$\sum_{a\in {\mathbb F_{q}}}D_{n,3}(1,a)$.
By Proposition 2.6, one has
\begin{align}
\sum_{n=0}^{\infty}D_{n,3}(1,x)t^n&=\frac{2t-1}{1-t+xt^2}
=\frac{2t-1}{1-t}\frac{1}{1-\frac{t^2}{t-1}x} \notag\\
&=\frac{2t-1}{1-t}\Big(1+\sum_{k=1}^{q-1}\sum_{\ell=0}^{\infty }
\bigg(\frac{t^2}{t-1}\bigg)^{k+\ell (q-1)}x^{k+\ell (q-1)}\Big) \notag\\
&\equiv\frac{2t-1}{1-t}\Big(1+\sum_{k=1}^{q-1}\sum_{\ell=0}^{\infty }
\bigg(\frac{t^2}{t-1}\bigg)^{k+\ell (q-1)}x^{k}\Big)\pmod{x^q-x} \notag\\
&=\frac{2t-1}{1-t}\Big(1+\sum_{k=1}^{q-1}
\frac{(\frac{t^2}{t-1})^k}{1-(\frac{t^2}{t-1})^{q-1}}x^k\Big) \notag\\
&=\frac{2t-1}{1-t}\Big(1+\sum_{k=1}^{q-1}
\frac{(t-1)^{q-1-k}t^{2k}}{(t-1)^{q-1}-t^{2(q-1)}}x^k\Big).\label{c6}
\end{align}
Moreover, by Theorem 2.4 (ii), it follows that for any
$x\in{\mathbb F_{q}}\setminus\{\frac{1}{4}\}$, one has
$$D_{n_1,3}(1,x)=D_{n_2,3}(1,x)$$
when $n_1\equiv n_2\pmod{q^2-1}$. Thus if $x\ne \frac{1}{4}$, one has
\begin{align}
\sum_{n=0}^\infty D_{n,3}(1,x)t^n&=1+\sum_{n=1}^{q^2-1}
\sum_{\ell=0}^{\infty }D_{n+\ell(q^2-1),3}(1,x)t^{n+\ell(q^2-1)}\notag\\
&=1+\sum_{n=1}^{q^2-1}D_{n,3}(1,x)\sum_{\ell=0}^{\infty }t^{n+\ell(q^2-1)}\notag\\
&=1+\frac{1}{1-t^{q^2-1}}\sum_{n=1}^{q^2-1}D_{n,3}(1,x)t^{n}.\label{c7}
\end{align}
Then (\ref{c6}) together with (\ref{c7}) gives that
for any $x\ne \frac{1}{4}$, we have
\begin{align}
&\sum_{n=1}^{q^2-1}D_{n,3}(1,x)t^n\\
=&\Big(\sum_{n=0}^{\infty}D_{n,3}(1,x)t^n-1\Big)(1-t^{q^2-1})\notag\\
\equiv & \Big(\frac{2t-1}{1-t}-1\Big)(1-t^{q^2-1})
+\frac{(1-t^{q^2-1})(2t-1)}{1-t}\sum_{k=1}^{q-1}
\frac{(t-1)^{q-1-k}t^{2k}}{(t-1)^{q-1}-t^{2(q-1)}}x^k\pmod{x^q-x}\notag\\
=&\frac{(3t-2)(1-t^{q^2-1})}{1-t}+h(t)\sum_{k=1}^{q-1}
(t-1)^{q-1-k}t^{2k}x^k,\label{c8}
\end{align}
where
$$h(t):=\frac{(t^{q^2-1}-1)(2t-1)}{(t-1)^q-(t-1)t^{2(q-1)}}.$$
\begin{lem}\cite{[LN]}
Let $u_0,u_1,\cdots,u_{q-1}$ be the list of the all elements of
${\mathbb F_{q}}$. Then
$$
\sum_{i=0}^{q-1}u_i^k={\left\{
  \begin{array}{rl}
0,& {\rm if} \ 0\le k\le q-2,\\
-1,& {\rm if} \ k=q-1.
\end{array}
\right.}
$$
\end{lem}

Now by Theorem 2.4 (i), Lemma 4.1 and (\ref{c8}), we derive that
\begin{align}
&\sum_{n=1}^{q^2-1}\sum_{a\in{\mathbb F_{q}}}D_{n,3}(1,a)t^n\notag\\
=&\sum_{n=1}^{q^2-1}D_{n,3}\Big(1,\frac{1}{4}\Big)t^n+\sum_{n=1}^{q^2-1}
\sum_{a\in{\mathbb F_{q}}\setminus\{\frac{1}{4}\}}D_{n,3}(1,a)t^n\notag\\
=&\sum_{n=1}^{q^2-1}\frac{3n-1}{2^n}t^n+\sum_{a\in{\mathbb F_{q}}
\setminus\{\frac{1}{4}\}}\frac{(3t-2)(1-t^{q^2-1})}{1-t}+h(t)\sum_{k=1}
^{q-1}(t-1)^{q-1-k}t^{2k}\sum_{a\in{\mathbb F_{q}}\setminus\{\frac{1}{4}\}}a^k\notag\\
=&\sum_{n=1}^{q^2-1}\frac{3n-1}{2^n}t^n+(q-1)\frac{(3t-2)(1-t^{q^2-1})}{1-t}
+h(t)\sum_{k=1}^{q-1}(t-1)^{q-1-k}t^{2k}\sum_{a\in{\mathbb F_{q}}}a^k\notag\\
&\ \ \ -h(t)\sum_{k=1}^{q-1}(t-1)^{q-1-k}t^{2k}\Big(\frac{1}{4}\Big)^k\notag\\
=&\sum_{n=1}^{q^2-1}\frac{3n-1}{2^n}t^n-\frac{(3t-2)(1-t^{q^2-1})}{1-t}-h(t)t^{2(q-1)}
-h(t)\sum_{k=1}^{q-1}(t-1)^{q-1-k}t^{2k}\Big(\frac{1}{4}\Big)^k.\label{c9}
\end{align}

Since $(t-1)^q=t^q-1$ and $q$ is odd, one has
\begin{align}
h(t)=&\frac{(t^{q^2-1}-1)(2t-1)}{(t-1)^q-(t-1)t^{2(q-1)}}\notag\\
=&\frac{(t^{q^2-1}-1)(2t-1)}{(1-t^{q-1})(t^q-t^{q-1}-1)} \notag\\
=&\frac{(t^{q^2}-t)(2t-1)}{(t-t^{q})(t^q-t^{q-1}-1)}\notag\\
=&\frac{(t^{q}-t)^q+t^q-t}{t-t^{q}}\cdot \frac{2t-1}{t^q-t^{q-1}-1}\notag\\
=&\frac{(-1-(t-t^q)^{q-1})(2t-1)}{t^q-t^{q-1}-1}\notag\notag\\
=&\frac{(2t-1)\sum_{i=0}^{q^2-q}b_it^i}{t^q-t^{q-1}-1},\label{c10}
\end{align}
where
$$\sum_{i=0}^{q^2-q}b_it^i:=-1-(t-t^q)^{q-1}.$$
Then by the binomial theorem applied to $(t-t^q)^{q-1}$,
we can derive the following expression for the coefficient $b_i$.
\begin{prop}
For each integer $i$ with $0\le i\le q^2-q$, write $i=\alpha+\beta q$
with $\alpha $ and $\beta$ being integers such that
$0\le \alpha,\beta \le q-1$. Then
$$b_i={\left\{
\begin{array}{ll}
(-1)^{\beta +1}\binom{q-1}{\beta},& {\it if} \ \alpha +\beta =q-1,\\
-1,& {\it if} \ \alpha =\beta =0,\\
0,& {\it otherwise}.
\end{array}
\right.}$$
\end{prop}
For convenience, let
$$a_n:=\sum_{a\in{\mathbb F_{q}}}D_{n,3}(1,a).$$
Then by (\ref{c9}) and (\ref{c10}), we arrive at
\begin{align}
&\sum_{n=1}^{q^2-1}\Big(a_n-\frac{3n-1}{2^n}\Big)t^n\\
=&-\frac{(3t-2)(1-t^{q^2-1})}{1-t}-\frac{(2t-1)\sum_{i=1}^{q^2-q}b_it^i}{t^q-t^{q-1}-1}
\Big(t^{2(q-1)}+\sum_{k=1}^{q-1}(t-1)^{q-1-k}t^{2k}\Big(\frac{1}{4}\Big)^k\Big),
\end{align}
which implies that
\begin{align}
&(t^q-t^{q-1}-1)\sum_{n=1}^{q^2-1}\Big(a_n-\frac{3n-1}{2^n}\Big)t^n\notag\\
=&-(t^q-t^{q-1}-1)(3t-2)\sum_{i=0}^{q^2-2}t^i-(2t-1)\Big(t^{2(q-1)}+\sum_{k=1}^{q-1}
(t-1)^{q-1-k}t^{2k}\Big(\frac{1}{4}\Big)^k\Big)\sum_{i=0}^{q^2-q}b_it^i. \label{c11}
\end{align}

Let
$$\sum_{i=1}^{q^2+q-1}c_it^i$$
denote the right-hand side of (\ref{c11}) and let
$$d_n:=a_n-\frac{3n-1}{2^n}$$
for each integer $n$ with $1\le n\le q^2-1$. Then (\ref{c11}) can be reduced to
\begin{equation}\label{c12}
(t^q-t^{q-1}-1)\sum_{n=1}^{q^2-1}d_nt^n=\sum_{i=1}^{q^2+q-1}c_it^i.
\end{equation}
Then by comparing the coefficient of $t^i$ with $1\le i\le q^2+q-1$
of the both sides in (\ref{c12}), we derive the following relations:
$${\left\{
  \begin{array}{ll}
c_j=-d_j,& {\rm if} \ 1\le j\le q-1,\\
c_q=-d_1-d_q,& \\
c_{q+j}=d_j-d_{j+1}-d_{q+j},& {\rm if} \ 1\le j\le q^2-q-1,\\
c_{q^2+j}=d_{q^2-q+j}-d_{q^2-q+j+1},& {\rm if} \ 0\le j\le q-2,\\
c_{q^2+q-1}=d_{q^2-1},&
\end{array}
\right.}$$
from which we can deduce that
\begin{align}\label{c13}
{\left\{
  \begin{array}{ll}
d_j=-c_j,& {\rm if} \ 1\le j\le q-1,\\
d_q=c_1-c_q,& \\
d_{\ell q+j}=d_{(\ell -1)q+j}-d_{(\ell-1)q+j+1}-c_{\ell q+j},
& {\rm if} \ 1\le \ell\le q-2\ \ {\rm and}\ \ 1\le j\le q-1,\\
d_{\ell q}=d_{(\ell -1)q}-d_{(\ell-1)q+1}-c_{\ell q},& {\rm if} \ 2\le \ell\le q-2,\\
d_{q^2-q+j}=\sum_{i=j}^{q-1}c_{q^2+i},& {\rm if} \ 0\le j\le q-1.
\end{array}
\right.}
\end{align}

Finally, (\ref{c13}) together with the following identity
$$\sum_{a\in {\mathbb F_{q}}}D_{n,3}(1,a)=d_n+\frac{3n-1}{2^n}$$
shows that the last main result of this paper is true:
\begin{thm}
Let $c_i$ be the coefficient of $t^i$ in the right-hand side of (\ref{c11})
with $i$ being an integer such that $1\le i\le q^2+q-1$. Then we have
\begin{align*}
&\sum_{a\in {\mathbb F_{q}}}D_{j,3}(1,a)=-c_j+\frac{3j-1}{2^j}\ \ if\ \ 1\le j\le q-1,\\
&\sum_{a\in {\mathbb F_{q}}}D_{q,3}(1,a)=c_1-c_q-\frac{1}{2},\\
&\sum_{a\in {\mathbb F_{q}}}D_{\ell q+j,3}(1,a)=\sum_{a\in {\mathbb F_{q}}}D_{(\ell-1)q+j,3}(1,a)-
\sum_{a\in {\mathbb F_{q}}}D_{(\ell-1)q+j+1,3}(1,a)-c_{\ell q+j}+\frac{3}{2^{\ell +j}}\\
&\ \ \ \ if\ \ 1\le \ell\le q-2 \ {\it and} \ 1\le j\le q-1,\\
&\sum_{a\in {\mathbb F_{q}}}D_{\ell q,3}(1,a)=\sum_{a\in {\mathbb F_{q}}}D_{(\ell-1)q,3}(1,a)-
\sum_{a\in {\mathbb F_{q}}}D_{(\ell-1)q+1,3}(1,a)-c_{\ell q}+\frac{3}{2^{\ell}}\ \ if\ \ 2\le \ell\le q-2
\end{align*}
and
\begin{align*}
&\sum_{a\in {\mathbb F_{q}}}D_{q^2-q+j,3}(1,a)
=\sum_{i=j}^{q-1}c_{q^2+i}+\frac{3j-1}{2^j}\ \ if\ \ 0\le j\le q-1.
\end{align*}
\end{thm}

\end{document}